\documentclass[11pt]{amsart}

\usepackage{times}

\usepackage{geometry}
\usepackage{amssymb}
\usepackage{latexsym, amsmath, amscd,amsthm}
\usepackage{graphicx}
\usepackage[percent]{overpic}
\usepackage{pdfsync}
\usepackage{units}
\usepackage{hyperref}
\usepackage{paralist}

\newtheorem{theorem}{Theorem}
\newtheorem{lemma}[theorem]{Lemma}
\newtheorem{proposition}[theorem]{Proposition}
\newtheorem{corollary}[theorem]{Corollary}

\theoremstyle{definition}
\newtheorem{definition}[theorem]{Definition}

\newtheorem{conjecture}[theorem]{Conjecture}

\def\thm#1{Theorem~\ref{thm:#1}}

\newcommand{\R}{\mathbb{R}}
\newcommand{\Z}{\mathbb{Z}}

\newcommand{\TC}{\operatorname{TC}}
\newcommand{\arc}[1]{\gamma_{#1}}

\newcommand{\Len}{\operatorname{Len}}

\newcommand\norm[1]{\left| #1 \right|}
\newcommand{\slq}{Slq}

\newcommand{\FTC}{\operatorname{FTC}}
\newcommand{\FTCWC}{\operatorname{FTCWC}}
\def\co{\colon\!}
\graphicspath{./figs}

\begin{document}
\title{Square-like quadrilaterals inscribed in embedded space curves}

\author{Jason Cantarella}
\address{Mathematics Department, University of Georgia, Athens GA 30602}
\email{jason.cantarella@uga.edu}
\author{Elizabeth Denne}
\address{Mathematics Department, Washington \& Lee University, Lexington VA 24450}
\email{dennee@wlu.edu}
\author{John McCleary}
\address{Mathematics \& Statistics Department, Vassar College, Poughkeepsie NY 12604}
\email{mccleary@vassar.edu}

\makeatletter								
\@namedef{subjclassname@2020}{%
  \textup{2020} Mathematics Subject Classification}
\makeatother

\subjclass[2020]{Primary 53A04, Secondary 55R80, 57Q65, 58A20, 51M04}
\keywords{Square-peg problem, square-like quadrilaterals, embedded space curves, finite total curvature.}

\begin{abstract} 
The square-peg problem asks if every Jordan curve in the plane has four points which are the vertices of a square. 
The problem is open for continuous Jordan curves, but it has  been resolved for various regularity classes of curves between continuous and $C^1$-smooth Jordan curves.  Here, in a generalization of the square-peg problem, we consider embedded curves in space, and ask if they have inscribed  quadrilaterals with equal sides and equal diagonals. We call these quadrilaterals ``square-like''.
We give a regularity class (finite total curvature without cusps) in which we can prove that every embedded curve has an inscribed square-like quadrilateral. The key idea is to use local data to show that short enough arcs have small curvature, thus ruling out small squares. This allows us to successfully use a limiting argument on approximating curves.  
 
\end{abstract}
\date{\today}
\maketitle


\section{Introduction}\label{sect:intro} 

Take an embedding $\gamma\co S^1\hookrightarrow \R^n$ of a circle in $\R^n$. An {\em inscribed polygon} is a polygon whose vertices lie on the curve. Note that when $n=2$, the sides of the polygon do not have to lie in the interior of the planar curve. There is a series of problems which ask what kind of quadrilaterals can be inscribed in such curves\footnote{Including recent results on finding inscribed rectangles and cyclic quadrilaterals; see for instance \cite{MR3810027, MR4061975, MR2038265, greene2020cyclic, hugelmeyer2018smooth, hugelmeyer2019inscribed, matschke2020quadrilaterals, schwartz2018rectangle, schwartz2019trichotomy}.}.  A classic example of such a problem is the old, yet still open, question due to O.~Toeplitz~\cite{Toeplitz}. He asked whether a square can be inscribed in a Jordan curve (a continuous, simple, closed curve in the plane).  If we think of the Jordan curve as the ``round hole'', this conjecture has affectionately been nick-named the {\em square-peg problem} by mathematicians.   There have been many attempts to resolve the square-peg problem, and a brief overview of the history of the problem can be found in Appendix~\ref{sect:history}, as well as a number of survey articles (see for instance~\cite{MR2492772, MR3184501, Pak-Discrete-Poly-Geom}).

\begin{figure}[t]
\hfill
\raisebox{-0.5\height}{\includegraphics[height=1.5in]{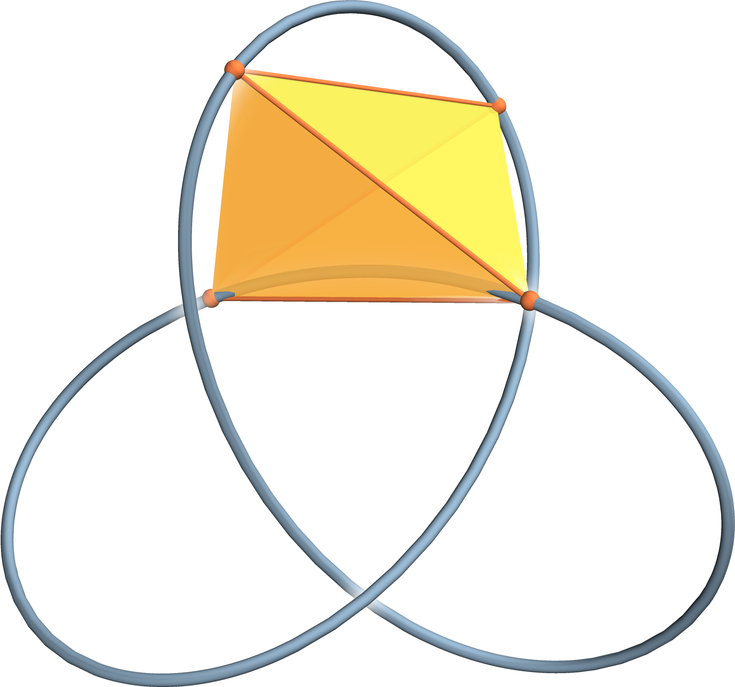}}
\hfill
\raisebox{-0.5\height}{\includegraphics[height=1.25in]{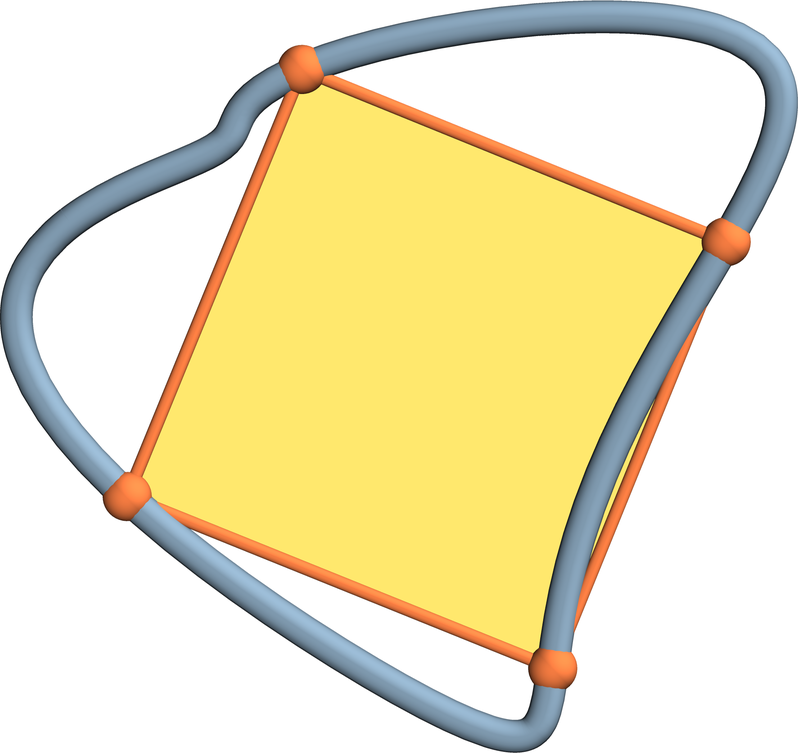}}
\hfill
\hphantom{.}
\caption{On the left, we see a square-like quadrilateral inscribed in a space curve. The four equal sides are marked in red, while the diagonals are unmarked. The shape forms a special tetrahedron. On the right, we see a square-like quadrilateral inscribed in a plane curve. This shows that a planar square-like quadrilateral is indeed a square. 
\label{fig:firstexample}}
\end{figure} 

If we think of a square as being a quadrilateral that has equal sides and equal diagonals, then we have a property of quadrilaterals that holds in any dimension. We say such a quadrilateral is {\em square-like} (see Definition~\ref{def:slq}).  In this paper, we look at a generalization of the square-peg problem:
\begin{conjecture} There is a square-like quadrilateral inscribed in any embedding of $S^1$ in $\R^n$. 
\label{conj:slq}
\end{conjecture}

Two examples of inscribed square-like quadrilaterals are shown in Figure~\ref{fig:firstexample}. Just as with the square-peg problem, we expect Conjecture~\ref{conj:slq} to be proved for various regularity classes of curves.  Indeed, in  \cite{Slq}, we proved that there is a dense family of $C^\infty$-smoothly embedded $S^1$ in $\R^n$, each of which has an odd number of inscribed square-like quadrilaterals. (This result and the Whitney $C^\infty$ topology used are reviewed in Section~\ref{sect:slq}.)  Going back to the square-peg problem, there are a number of solutions which required a mild-regularity hypothesis, for example those of W.~Stromquist~\cite{MR1045781}, B.~Matschke~\cite{phd-Matschke} and T.~Tao~\cite{MR3731730}. These hypotheses require the curve to be planar, and the regularity of the embedding lies  somewhere between continuous and $C^1$-smooth.  In this paper, we find a regularity condition of a similar flavor that allow us to resolve Conjecture~\ref{conj:slq}, but our regularity condition  (discussed below) holds in any dimension.

The key idea in this paper is to look at a limiting argument. That is, take any curve which is a limit of 
a sequence of $C^\infty$-smooth curves with inscribed square-like quadrilaterals. Can we show that the limiting curve also has an inscribed square-like quadrilateral?
The problem is clear: The sequence of square-like quadrilaterals on the approximating curves may have sidelengths approaching zero. If one could construct a general lower bound on these sidelengths in terms of the global geometry of the curves with the inscribed square-like quadrilaterals, this possibility could be ruled out. We do not know of any explicit example of a family of curves where all the inscribed square-like quadrilaterals have sidelengths converging to zero, so this approach may yet be possible. However, this line of attack has been more or less obvious from the start, and nobody has managed to construct such an argument in the past century.

Our considerably more modest goal in this paper is to rule out small square-like quadrilaterals using \emph{local}, rather than global, data about the limit curve, and in this way to extend our results to the class of curves of \emph{finite total curvature without cusps} ($\FTCWC$). In Section~\ref{sect:ftc}, we review the definition of what it means for a curve to be of finite total curvature, and explore what it means for a sequences of curves to converge uniformly in position, arclength and total curvature (Definition~\ref{def:uniform in pos len tc}). As well as allowing us to resolve Conjecture~\ref{conj:slq}, the appeal of this choice of regularity class is largely that the class of curves of finite total curvature is a well-understood space~(cf.~\cite{math.GT/0606007}). 

Our argument is completed in three parts in Section~\ref{sect:curve-pi}. First, we show that each embedded curve $\gamma$ in $\FTCWC$ has no inscribed square-like quadrilaterals with side length smaller than a positive constant, denoted by $\operatorname{\pi-d}(\gamma)$. 
Next, we show that if there is a sequence of curves $\gamma_i$ converging to $\gamma$ in position, arclength and total curvature, and $\operatorname{\pi-d}(\gamma)>0 $, then $\operatorname{\pi-d}(\gamma_i)>0$ too. 
Finally, we take any $\FTCWC$ embedding $\gamma$ of $S^1$ in $\R^n$. We approximate $\gamma$ by a sequence of smooth $\FTC$ embeddings $\gamma_i$ of $S^1$ in $\R^n$, each of which has an odd number of square-like quadrilaterals.   The first two steps then imply that the corresponding sequence of square-like quadrilaterals has a convergent subsequence with limit a square-like quadrilateral with sidelength at least $\operatorname{\pi-d}(\gamma)$.
This give us our main result in \thm{FTCWC}: Any embedding of $S^1$ in $\R^n$ that is in $\FTCWC$ has an inscribed square-like quadrilateral.

One immediate consequence of  \thm{FTCWC} is that we resolve the square-peg problem for (planar) Jordan curves which are in $\FTCWC$. How does this regularity class compare to the other solutions of the square-peg problem? Stromquist~\cite{MR1045781} and Matschke~\cite{phd-Matschke,MR3184501} give local conditions for planar curves that are weaker than ours. However, unlike our results, they do not discuss embeddings of circles in $\R^n$ for $n>2$. Tao~\cite{MR3731730} gives a different regularity condition. His condition is for Jordan curves which are the union of two Lipschitz graphs that agree at the endpoints, and whose Lipschitz constants are strictly less than one. Such curves might be thought of intuitively as ``vertically star-like''.  Tao's class of curves and the class of FTCWC curves intersect one another, but neither class of curves fully contains the other.  


\section{Topology and square-like quadrilaterals}
\label{sect:slq}

We use compactified configuration spaces and multijet transversality in \cite{Slq} to prove that there is a dense family of smoothly embedded circles in $\R^n$, each of which has an odd number of inscribed square-like quadrilaterals.  

\begin{definition} \label{def:slq} A {\em square-like quadrilateral} in $\R^n$ is a quadrilateral $pqrs$ with equal sides ($|pq|=|qr|=|rs|=|sp|$) and equal diagonals ($|pr|=|qs|$). The set of all square-like quadrilaterals in $\R^n$ is denoted by $\slq$.
\end{definition} 
The structure of $\slq$ is discussed in detail in \cite{Slq}. We observe that when $n=2$, a square-like quadrilateral is a square.  When $n\geq 3$, a square-like quadrilateral is a tetrahedron, and is also known as a tetragonal disphenoid\footnote{By definition, a tetragonal disphenoid is a tetrahedron with  four congruent isosceles triangle faces.}.   In \cite{Slq}, we prove the following.

\begin{theorem}[See Theorem 35 \cite{Slq}]
Take any smooth embedding of a curve $\gamma\co S^1\hookrightarrow \R^n$. Then there is a $C^\infty$-open neighborhood of $\gamma$ in $C^\infty(S^1,\R^n)$, in which there is, for all $m$, a $C^m$-dense set of smooth embeddings $\gamma'\co S^l\hookrightarrow \R^n$, each of which has an odd, finite set of inscribed square-like quadrilaterals

\label{thm:squarepeg}
\end{theorem}
  
In order to fully appreciate \thm{squarepeg}, we recall the topology of the spaces we are working in. In general, for manifolds $M$ and $N$, the space $C^\infty(M,N)$ has the Whitney $C^\infty$-topology (see for instance \cite{Hirsch}).   The sets of the form $$\mathcal{N}^m(f; (U,\phi), (V,\psi), \delta)$$ give a subbasis for the Whitney $C^m$-topology on $C^m(M,N)$ (where $m$ is finite). 
This subbasis is the subset of functions $g\colon M\to N$ that are smooth, and for coordinate charts $\phi\colon (U' \subset M) \to (U\subset \R^m)$ and $\psi\colon 
(V' \subset N) \to (V\subset \R^k)$ and $K \subset U$ compact with $g(\phi(K)) \subset V'$, then, for all $s \leq r$, and all $x \in \phi(K)$, 
$$\| D^s (\psi g \phi^{-1})(x) - D^s(\psi f \phi^{-1})(x)\| < \delta. $$
Here, $D^s F$ for a function $F \colon (U\subset \R^m) \to (V \subset \R^k)$ is the $k$-tuple of the $s$th  homogeneous parts of the Taylor series representations of
the projections of $F$.   Finally, the subspace $C^\infty(M,N)$ has the Whitney $C^\infty$-topology by taking the union of all subbases for all $m\geq 0$.

\section{Finite Total Curvature Curves}
\label{sect:ftc}

We recall some standard facts about curves of finite total curvature~(see for instance \cite{math.GT/0606007, MR0176402}). The \emph{total curvature} of a curve is the supremum of the total turning angles of all polygons inscribed in the curve. If this supremum is finite, we say the curve has \emph{finite total curvature} or is in $\FTC$. 
Curves in $\FTC$ have a number of desirable properties. They are always rectifiable, and so can be parametrized by arclength. They are almost everywhere differentiable, and a curve in $\FTC$ has one-sided tangent vectors at every point. In fact, these tangents differ only at countably many corner points. There is a Radon measure $\kappa$ on every $\gamma$ in $\FTC$ whose mass on any open subarc of $\gamma$ is the total curvature (in the above sense) of the subarc. This measure has a countable number of atoms at corners of the curve $\gamma$. The mass of each atom is the turning angle between these vectors. If this turning angle is $\pi$, we say the corner is a \emph{cusp}. 

Since $\FTC$ curves have a second derivative (at least weakly) it is natural to want to approximate them ``in $C^2$'' by smooth curves. Unfortunately, this is not quite possible. Note that the tangent indicatrix to an $\FTC$ curve has gaps at the corners of the curve, while the tangent indicatrix of a smooth curve forms a continuous curve on $S^{n-1}$. Thus the tangent vectors to a sequence of smooth curves approximating an $\FTC$ curve can't converge to tangents of the $\FTC$ curve near a corner of the $\FTC$ curve. However, we can come very close to a $C^2$ approximation in the following sense:

\begin{definition}
Suppose $\gamma\co \R\rightarrow \R^n$ is a curve in $\FTC$. Let $\Len(\gamma,a,b)$   be the length of the arc of $\gamma$ between $\gamma(a)$ and $\gamma(b)$ and $\kappa(\gamma,a,b)$ be the total curvature of this arc. We say that a sequence of finite total curvature curves $\gamma_i\co\R\rightarrow \R^n$ approximate $\gamma$ \emph{uniformly in position, arclength, and total curvature} if there are parametrizations of the $\gamma_i$ so that for each $\epsilon > 0$ there exists an $N$ so that for all $i > N$, we have the following:
\begin{enumerate}
\item For any $a$, $\norm{\gamma_i(a) - \gamma(a)} < \epsilon$.
\item For any $[a,b]$, $\norm{\Len(\gamma_i,a,b) - \Len(\gamma,a,b)} < \epsilon$.
\item For any $[a,b]$, $\norm{\kappa(\gamma_i,a,b) - \kappa(\gamma,a,b)} < \epsilon$.
\end{enumerate}
\label{def:uniform in pos len tc}
\end{definition}

\begin{proposition} 
\label{prop:approx}
Any $\FTC$ curve $\gamma$ may be approximated uniformly in position, arclength, and total curvature by smooth $\FTC$ curves $\gamma_i$.
\end{proposition}

\begin{proof}
This is an assembly of standard results about $\FTC$ curves. If we inscribe polygons with vertices equally spaced by arclength in $\gamma$, and parametrize them compatibly (so that the vertices have the same parameter values on $\gamma$ and on each polygon), the polygons converge uniformly in position and total curvature (cf. Lemma~4.2 of \cite{MR3048513}) and are all finite-total curvature curves (since their total curvatures are bounded by that of $\gamma$). 

To see that they converge uniformly in arclength, fix an arc $(a,b)$ of $\gamma$, and observe that the corresponding arcs of the $\gamma_i$ have bounded total curvature, and converge to the arc of $\gamma$ in Fr\'echet distance because they converge in position. Then use Theorem~5.1 of \cite{MR0176402} (see also \cite{MR2352603}) which states that for any rectifiable curves $K$, $L$,
\begin{equation*}
|\Len(K) - \Len(L)| \leq \delta(K,L) (\pi \max\{\TC(K),\TC(L)\} + 2)
\end{equation*}
where $\delta(K,L)$ is the Fr\'echet distance between $K$ and $L$. Note that this theorem is not obvious: it says that the standard examples of curves which converge in Fr\'echet distance but \emph{not} in arclength, such as a stairstep curve converging to the diagonal of a square, must all have unbounded total curvature. 

To finish the proof, smooth each polygon by rounding off corners-- the smooth curves have the same total curvature as the polygons (and are hence $\FTC$) and are close to the original polygons in position, arclength, and total curvature, as required.
\end{proof}
\section{Curvature $\pi$ is required to have an inscribed square-like quadrilateral }
\label{sect:curve-pi}

The definition of total curvature means that the total curvature of any curve $\gamma$, is at least as large as the total curvature of any polygonal curve inscribed in $\gamma$.  We can apply this to curves which have an inscribed square-like quadrilateral $pqrs$.

 Notice that if a square-like quadrilateral  $pqrs$ is inscribed in an arc of $\gamma$, the total curvature of the arc $\arc{pqrs}$ must be at least as large as the total curvature (or total turning angle) of the inscribed open square-like quadrilateral $pqrs$. Here, by {\em open square-like quadrilateral (polygon)}, we mean the open arc $\sigma_{pqrs}$ (or $pqrs$ without side $sp$). We compute the the total curvature of the open square-like quadrilateral $\sigma_{pqrs}$ by computing the turning angle at vertices $q$ and $r$.  When $pqrs$ is a planar square, we see the open square  has total turning angle $\pi$.  We now prove that the turning angle is at least $\pi$ if $pqrs$ is an open square-like quadrilateral.

\begin{lemma} \label{lem:sqrturn} Any square-like quadrilateral  $pqrs$ has the property that the open square-like quadrilateral $\sigma_{pqrs}$ has total curvature $\kappa(\sigma_{pqrs}) \geq \pi$, with equality if and only if $pqrs$ is a planar square.  
\end{lemma}

\begin{proof} 
Consider the square-like quadrilateral found in Figure~\ref{fig:tetraturning} where $pqrs$ has equal sides $pq$, $qr$, $rs$, and $sq$ and equal diagonals $pr$ and $qs$.
We may assume without loss of generality that the sides have length 1. We construct the midpoint $t$ of $qs$. Since $\triangle pqs$ is isosceles, we can conclude that $\angle ptq$ is right and that $\angle qpt=\angle spt=\theta$.
We then have $pt = \cos \theta$ and $qt = \sin \theta$. Further, since $qs = pr$, we have $pr = 2 \sin \theta$.

\begin{figure}[htbp]
\begin{overpic}{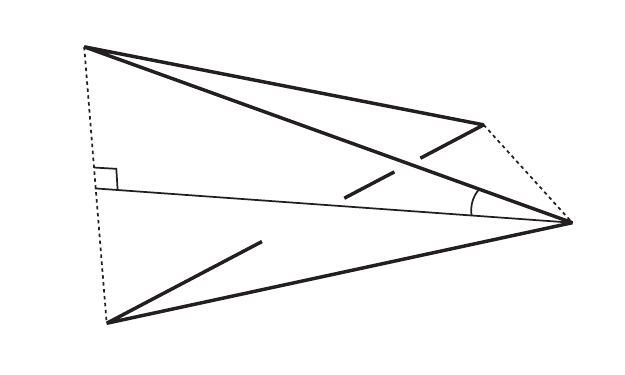}
\put(92.5,22){$p$}
\put(10,52){$q$}
\put(13.8,6){$s$}
\put(78,41){$r$}
\put(42,22.2){$\cos \theta$}
\put(2,39){$\sin\theta$}
\put(12,29){$t$}
\put(70,26.5){$\theta$}
\end{overpic}
\caption{The open arc $\sigma_{pqrs}$ of the square-like quadrilateral $pqrs$ shown has total curvature given by $2\pi - 4\theta$. We observe, however, that $pt$ has length $\cos \theta$ and $qt$ has length $\sin \theta$, while $2qt = qs = pr$ is less than $2pt$. Thus $\cos \theta \geq \sin \theta$ and $\theta \leq \pi/4$.}
\label{fig:tetraturning}
\end{figure}
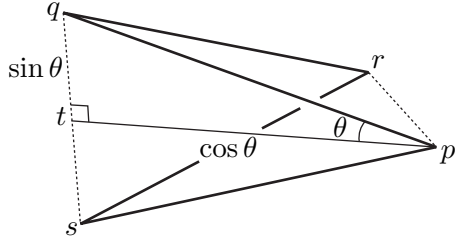 

Since $pq = rq$,  $ps = rs$, and $qs$ is a common side,  we have $\triangle rqs \cong \triangle pqs$. Thus $\angle qrs = \angle qps = 2 \theta$ and,  after joining $r$ to $t$, we find $rt = \cos \theta$ as above.  So by the triangle inequality (on $\triangle ptr$) we have $pt + tr \geq pr$, or
\begin{equation*}
2 \cos \theta \geq 2 \sin \theta.
\end{equation*}
This means that $\theta \leq \pi/4$, and $\theta = \pi/4$ if and only if $t$ is on the line $pr$. In this case $pqrs$ is planar (and hence it is a square). Now, at vertices $q$ and $r$, the turning angle of $pqrs$ is $\pi - 2\theta$. Thus the total turning angle of the open square-like quadrilateral $\sigma_{pqrs}$ is $2 \pi - 4\theta \geq \pi$, as desired.
\end{proof}

\subsection{On a finite-total curvature curve without cusps, short enough arcs have small curvature}
\label{sect:FTCWC}

In this section we will restrict our attention to curves with finite total curvature, but without cusps. We say such a curve is in $\FTCWC$. Recall that a cusp is a point on a curve with turning angle $\pi$. What is our goal? We will take an embedded curve in $\FTCWC$, and a sequence of smooth FTC embedded curves $\gamma_i$ converging to $\gamma$. We will prove that there exists a constant $c > 0$ so that no square-like quadrilateral inscribed in $\gamma_i$ has sidelength less than $c$. 

\begin{definition}
\label{def:par}
We define the \textbf{$\pi$-distance} of an FTC curve $\gamma$, denoted $\operatorname{\pi-d}(\gamma)$.
The value $\ell$ is an {\it admissible distance bound} if every open subarc $(a,b)$ of $\gamma$ with $\norm{\gamma(a) - \gamma(b)} < \ell$ has $\kappa(\gamma,a,b) < \pi$. Then 
\begin{equation*}
\operatorname{\pi-d}(\gamma) =  \sup_{\ell \text{ is admissible}} \ell  = 
\inf_{\ell \text{ is inadmissible}} \ell.
\end{equation*}
\end{definition}

Note that if $\ell$ is inadmissible, then there is some subarc $(a,b)$ with $\norm{\gamma(a) - \gamma(b)} < \ell$, but $\kappa(\gamma,a,b) \geq \pi$. The point of $\operatorname{\pi-d}(\gamma)$ is that it provides a lower bound on the side length of a square-like quadrilateral  inscribed in $\gamma$.

\begin{lemma}\label{lem:pi d}
Any square-like quadrilateral  inscribed in an FTC curve $\gamma$ has sidelength greater than or equal to $\operatorname{\pi-d}(\gamma)$. 
\end{lemma}

\begin{proof}
Let $pqrs$ be an inscribed square-like quadrilateral  in $\gamma$, and consider the open polygon $\sigma_{pqrs}$ which has end-to-end distance $\norm{\gamma(p) - \gamma(s)}$. By Lemma~\ref{lem:sqrturn}, the square-like quadrilateral  is an inscribed open polygon with total curvature at least $\pi$. Thus $\kappa(\gamma,p,s) \geq \pi$. This means that $\norm{\gamma(p) - \gamma(s)}$ is an inadmissible distance bound, and hence it is at least $\operatorname{\pi-d}(\gamma)$, as desired.
\end{proof}

We now want to show that an embedded curve in $\FTCWC$ is the limit of a sequence of smooth curves with inscribed square-like quadrilaterals with side lengths uniformly bounded above zero. We proceed in two steps: first we will show that $\gamma$ itself has $\operatorname{\pi-d}$ bounded above, then that $\operatorname{\pi-d}$ behaves nicely under the sort of convergence of curves we introduced above in Definition~\ref{def:uniform in pos len tc}.

\begin{lemma}
\label{lem:padgzero}
If $\gamma$ is an embedded curve in $\FTCWC$, then $\operatorname{\pi-d}(\gamma) > 0$.
\end{lemma}

\begin{proof} 
Suppose not. Since $\operatorname{\pi-d}(\gamma) = 0$, there is a sequence of inadmissible $\ell_i \rightarrow 0$. So there exists a collection of open subarcs $A_i$ of $\gamma$ whose endpoints $a_i$, $b_i$ have $\norm{\gamma(a_i) - \gamma(b_i)} \rightarrow 0$, while $\kappa(\gamma,a_i,b_i) \geq \pi$. Passing to a subsequence where $a_i \rightarrow a$ and $b_i \rightarrow b$, we see that $\gamma(a) = \gamma(b)$, and hence $a = b$ because $\gamma$ is embedded. 

Now as the $A_i$ approach $\{a\}$, their total curvature $\kappa(A_i) \geq \pi$. Since $\gamma$ is compact, we may pass to a subsequence of $A_i$ that are nested and converge to a point $p$. Since $\kappa$ is an outer-regular measure, this means that $\kappa(p) \geq \pi$. Since $\kappa(p)$ is a turning angle, it is always $\leq \pi$. Thus $\kappa(p) = \pi$ and $p$ is a cusp point, contradicting our assumption that $\gamma$ was in $\FTCWC$.
\end{proof}

Since $\operatorname{\pi-d}$ is defined by lengths, distances, and curvatures, we can expect it to behave nicely as we take limits in the sense of Definition~\ref{def:uniform in pos len tc}.

\begin{proposition}
If $\gamma_i \rightarrow \gamma$ uniformly in position, arclength, and total curvature in the sense of Definition~\ref{def:uniform in pos len tc}, and $\operatorname{\pi-d}(\gamma) > 0$, then $ \liminf_{i\rightarrow \infty} \operatorname{\pi-d}(\gamma_i) > 0$.
\label{prop:bounded pi d}
\end{proposition}

\begin{proof}
Suppose not. For any $\epsilon > 0$, there must be infinitely many $\gamma_i$ with $\operatorname{\pi-d}(\gamma_i) < \epsilon$. Each $\gamma_i$ contains a subarc $(a_i,b_i)$ with $\norm{\gamma_i(a_i) - \gamma_i(b_i)} < \epsilon$, but $\kappa(\gamma_i,a_i,b_i) \geq \pi$. By compactness, we can assume that we have passed to a  subsequence where $a_i \rightarrow a$ and $b_i \rightarrow b$. 

Now by convergence in position, $\norm{\gamma(a) - \gamma(b)} \leq \epsilon$. Let us expand the open arc $(a,b)$ of $\gamma$ slightly to an open subarc $(a',b')$ with $\norm{\gamma(a') - \gamma(b')} \leq 3 \epsilon$, say, and again pass to a subsequence where $(a_i,b_i) \subset (a',b')$ for all $i$. Now for any $\delta > 0$, by convergence in total curvature, for large enough $i$ we have
\begin{equation*}
\norm{\kappa(\gamma,a',b') - \kappa(\gamma_i,a',b')} < \delta
\end{equation*}
so
\begin{equation*}
\kappa(\gamma,a',b') > \kappa(\gamma_i,a',b') - \delta \geq \kappa(\gamma_i,a_i,b_i) - \delta \geq \pi - \delta.
\end{equation*}
where $\kappa(\gamma_i,a',b') \geq \kappa(\gamma_i,a_i,b_i)$ because $(a_i,b_i) \subset (a',b')$. Since $\delta$ was arbitrary, this proves that $\kappa(\gamma,a',b') \geq \pi$. 

However, this means that $3 \epsilon > \norm{\gamma(a') - \gamma(b')}$ is an inadmissible distance bound for $\gamma$, and hence that $\operatorname{\pi-d}(\gamma) < 3 \epsilon$. Since $\epsilon$ was arbitrary, this proves that $\operatorname{\pi-d}(\gamma) = 0$, providing the required contradiction.
\end{proof}

We are ready to construct an inscribed square-like quadrilateral  on any $\FTCWC$ curve. We have done all the hard work above; it remains only to assemble the component pieces.

\begin{theorem}
\label{thm:FTCWC}
Let $\gamma\co S^1\hookrightarrow \R^n$ be an embedding of $S^1$ in $\R^n$. If $\gamma$ is in $\FTCWC$, then $\gamma$ has an inscribed square-like quadrilateral. 
\end{theorem}

\begin{proof}
First, we may approximate $\gamma$ by a sequence of smooth FTC curves $\gamma_i$ with convergence in position, arclength, and total curvature by Proposition~\ref{prop:approx}.  Note that since $\gamma_i$ are smooth, they are automatically in $\FTCWC$.
By making a $C^2$-small perturbation of each $\gamma_i$, we may assume by Theorem~\ref{thm:squarepeg} that each $\gamma_i$ contains at least one inscribed square-like quadrilateral\footnote{\thm{squarepeg} says there is a $C^\infty$-open neighborhood of each $\gamma_i$ in which there is, for all $m$ (in particular $m=2$), a $C^m$-dense set of embeddings, each with an odd number of  square-like quadrilaterals.}. Since our perturbations were $C^2$-small, the sequence of curves $\gamma_i$ still enjoys finite total curvature and converges to $\gamma$ in position, arclength, and total curvature.

By Lemma~\ref{lem:padgzero} and Proposition~\ref{prop:bounded pi d}, there is a constant $c>0$ so that we may pass to a subsequence of $\gamma_i$, each of which has $\operatorname{\pi-d}(\gamma_i) > c$. By Lemma~\ref{lem:pi d} the inscribed square-like quadrilateral  on each $\gamma_i$ has sidelength at least $c$. This is the crucial point in the proof: by bounding the sidelengths of these square-like quadrilaterals from below, we have ensured that they do not shrink away as we approach the limiting curve $\gamma$. 

From here, the argument is standard. We may assume that the inscribed square-like quadrilaterals in the $\gamma_i$ lie in a compact subset of  $\slq$, and hence that they have a convergent subsequence. The limit of this subsequence is a square-like quadrilateral  inscribed in the limit curve $\gamma$. 
\end{proof}

Since $C^2$ smooth curves are in $\FTCWC$, we immediately find:

\begin{corollary}\label{cor:C2-smooth}
Let $\gamma\co S^1\hookrightarrow \R^n$ be an embedding of $S^1$ in $\R^n$. If $\gamma$ is $C^2$-smooth, then  $\gamma$ has an inscribed square-like quadrilateral.
\end{corollary}

Since square-like quadrilaterals are squares in the plane, then we have also shown that all embeddings of the circle in the plane which are in $\FTCWC$ or which are $C^2$-smooth have an inscribed square. 

Note that we have lost something here. Our proof assumed that the limiting curves $\gamma_i$ had an odd number of square-like quadrilaterals. It is entirely possible that multiple square-like quadrilaterals coincide on the limiting curve $\gamma$, so the count of inscribed square-like quadrilaterals may no longer be odd.  Indeed there are explicit examples~\cite{2008arXiv0810.4806P, MR2798015, vanheijst2014algebraic} which show that for any $n$, there are curves with $n$ inscribed squares.

Also note that there exist smooth curves that are not FTC; these do not have corners but have spirals where curvature diverges. For these curves, Theorem~\ref{thm:squarepeg} still holds, but we can not conclude from Theorem~\ref{thm:FTCWC} that there is at least one square-like quadrilateral. (The spirals prevent the arguments of Proposition~\ref{prop:bounded pi d} from holding.)

Finally, we still do not know if Conjecture~\ref{conj:slq} holds for continuous embeddings of circles in $\R^n$. We have the intuition that if the square-peg problem holds true for continuous Jordan curves, then Conjecture~\ref{conj:slq} will hold true for continuous embeddings as well. Proving either conjecture needs new techniques which have yet to be developed.


\section*{Acknowledgements}
The authors would like to first thank Gerry Dunn who introduced us to the problem. We would also like to thank the people who have discussed the problem with us over the years: Jordan Ellenberg, Richard Jerrard, Rob Kusner, Benjamin Matschke,  Igor Pak, Strashimir Popvassiliev, John M. Sullivan, Cliff Taubes, and Gunter Ziegler.


\bibliography{ftcwc}{}
\bibliographystyle{plain}


\appendix
\section{A brief history of the square-peg problem}
\label{sect:history}

Let us recall the {\em square-peg problem}, originally due to O.~Toeplitz~\cite{Toeplitz}.

\begin{conjecture} Let $\gamma\co S^1\hookrightarrow \R^2$ be a Jordan curve (a continuous, simple, closed curve in plane). Then $\gamma(S^1)$ has an inscribed square.
\label{conj:squarepeg}
\end{conjecture}

While Conjecture~\ref{conj:squarepeg} is still open, here have been many attempts made to solve it. The interested reader can find a number of surveys of the history of these attempts, for example~\cite{MR2492772, MR3184501, Pak-Discrete-Poly-Geom}. This Appendix gives an overview of this history. We note that the majority of the existing solutions to Conjecture~\ref{conj:squarepeg} require that the Jordan curve be sufficiently regular.

In 1913,  A.~Emch~\cite{MR1506193,MR1506239,MR1506274}, then  K.~Von Zindler~\cite{MR1549095} in 1921, and C.M.~Christensen~\cite{MR0039284} in 1950, all proved the square-peg problem for convex closed curves.

In 1929, L.G.~Schnirel'man~\cite{Schnirelman:1944wf} (published in 1944), and then H. Guggenheimer \cite{MR188898} in 1965, both proved the square-peg problem for curves of continuous curvature of bounded variation (a class slightly larger than $C^2$-smooth). 

In 1961,  R.~Jerrard \cite{Jerrard}  showed that all analytic curves must have an odd or infinitely many inscribed squares. Earlier, O.~Frink, C.S.~Ogilvy~\cite{MR1527610} in 1950 proved the square-peg problem assuming some kind of smoothness (though this was not stated explicitly).

In 1989, W.~Stromquist \cite{MR1045781} proved the square-peg problem for $C^1$-smooth curves, as well as ones that are locally monotone\footnote{Locally monotone means that if every point $p$ of the curve has a neighborhood $U(p)$ and a direction $n(p)$ such that no chord of the curve is contained in $U(p)$ and parallel to $n(p)$.}. The latter class includes curves that are convex, or are polygonal, or piecewise $C^1$-smooth without cusps, or even certain curves which are nowhere differentiable.

In 1990,  H.B. Griffiths \cite{MR1095236} proved the square-peg problem for $C^2$-smooth (or higher); though this paper appears to contain serious errors (see~\cite{Slq,phd-Matschke}).

In 1991, V.~Klee and S.~Wagon \cite{MR1133201}  proved the square-peg problem or curves which star-like or symmetric 
around a point $z$ (specifically, where every line through $z$ meets the curve in exactly two points). A little later, in 1995, M.J.~Nielsen and S.E.~Wright \cite{MR1340790}  proved the square-peg problem for curves which are centrally symmetric though a point, or symmetric through reflection across a line. More recently in 2015, E.~Kelley (a student of Francis Su) wrote an honors thesis \cite{Kelley-honors} where she explicitly showed that a square is inscribed in the Koch snowflake curve (a centrally symmetric curve).
 
In 2008,  I.~Pak \cite{pak2008discrete} proved the square-peg problem for generic piecewise linear curves (with an elementary proof).

In 2011, B.~Matschke \cite{phd-Matschke,MR3184501} proved the square-peg problem for curves which do not contain small trapezoids of a certain type. It turns out that such curves form an open and dense subset of the space of embeddings $S^1\hookrightarrow \R^2$ with respect to the compact-open topology. Matschke also gives one of the few known global conditions on curves guaranteeing the existence of squares. Here, the curves must be contained in  an annulus with a certain ratio between the outer and inner radii, and which have a nontrivial winding number around the center of the annulus. In particular, this global condition does not require that the continuous curves are injective.

In 2014 (revised 2021), J.~Cantarella, E.~Denne, and J.~McCleary~\cite{Slq}  proved that an open dense set  (with respect to the Whitney $C^\infty$ topology) of $C^\infty$-smooth embeddings of $S^1$ in $\R^n$ have an odd number of squares.
   
 In 2017,  T.~Tao \cite{MR3731730} proved the square-peg problem for curves which are the union of two Lipschitz graphs that agree at the endpoints, and whose Lipschitz constants are strictly less than one. Such curves might be thought of intuitively as ``vertically star-like''.
 
Finally, this paper shows that the square-peg problem is proved for curves which are of finite total curvature and without cusps.

There are other closely related versions of the square-peg problem. For example, there are two slightly different discrete versions of the square-peg problem. The first, described by F.~Sagols and R.~Mar\'in \cite{MR2670395, MR3201252} in 2009, is where the vertices of polygonal Jordan curves are assumed to lie on a planar integer lattice grid $(\Z\times \R)\cup(\R \times \Z)$.  They then examine what kind of lattice curves have inscribed squares with integer coordinates. In 2014, V.H.~Pettersson {\em et al.} \cite{MR2798015} looked at curves where both vertices and edges lie on the planar integer lattice. They used computational methods to show that for $n\leq 13$, the side length of the largest square with edges on an $n\times n$ grid is at least $\nicefrac{1}{\sqrt{2}}$ times the side length of the largest axis-aligned square contained in the curve.  

There are also a number of results which count the number of squares. While several results \cite{Slq,Jerrard} show that we expect there to be generically an odd number of squares, we can not expect this in general. Indeed there are smooth convex curves (Popvassiliev~\cite{2008arXiv0810.4806P}) and piecewise linear curves (Sagols and Mar\'in~\cite{MR2798015}) which have exactly $n$ inscribed squares for any $n$. In addition, W.~van Heijst~\cite{vanheijst2014algebraic} proved that any real algebraic curve of degree $n$ in $\R^2$ inscribes either infinitely many squares or at most $\nicefrac{n^4-5m^2+4m}{4}$ squares.

\end{document}